\documentclass[english, letterpaper, 10 pt, conference]{article}

\usepackage[T1]{fontenc}
\usepackage{cite}
\usepackage{color}
 
\usepackage{amsmath}
\usepackage{amsthm}
\usepackage{amssymb}

\theoremstyle{plain}
\newtheorem{thm}{\protect\theoremname}
\theoremstyle{plain}
\newtheorem{lem}[thm]{\protect\lemmaname}

\newcommand{\mnote}[1]{}

\usepackage{algorithm,algorithmic}

\makeatother

\usepackage{babel}
\providecommand{\lemmaname}{Lemma}
\providecommand{\theoremname}{Theorem}

\begin{document}
\title{QPDAS: Dual Active Set Solver for\\Mixed Constraint
Quadratic Programming}
\author{Mattias F{\"a}lt, Pontus Giselsson\thanks{The authors are with the Department of Automatic Control, Lund University, Sweden. Both authors are financially supported by the Swedish Foundation for Strategic Research and the Swedish Research Council.
E-mail: \texttt{\{mattiasf,pontusg\}@control.lth.se}.
Accepted to 58th IEEE Conference on Decision and Control (CDC) 2019.
\textcopyright 2019 IEEE.  Personal use of this material is permitted.  Permission from IEEE must be obtained for all other uses, in any current or future media, including reprinting/republishing this material for advertising or promotional purposes, creating new collective works, for resale or redistribution to servers or lists, or reuse of any copyrighted component of this work in other works.
}}

\maketitle
\begin{abstract}
We present a method for solving the general mixed constrained convex
quadratic programming problem using an active set method on the dual
problem. The approach is similar to existing active set methods, but
we present a new way of solving the linear systems arising in the
algorithm. There are two main contributions; we present a new way
of factorizing the linear systems, and show how iterative refinement
can be used to achieve good accuracy and to solve both types of sub-problems
that arise from semi-definite problems.\\\\\mnote{ Rewrite parts of abstract}
\end{abstract}

\section{Introduction}

\mnote{Quadratic programs are essential {\color{red}BLA BLA BLA}. }Quadratic programming
has been studied extensively and many mature methods and algorithms
exist. The main approaches to solving these problems are interior
point~\cite{Wachter2006}, active set~\cite{Gill1995},\cite{Ferreau2014}, and operator-splitting
methods~\cite{Stellato2017},\cite{ODonoghue2016}.
Interior point methods typically converge
in a few iterations, but the computational complexity often makes
them impractical for large scale problems. Operator-splitting
methods, e.g. ADMM, are designed for cheap iterations,
but the convergence rate is usually much slower. This can be acceptable
when a low accuracy solution is sufficient, but for higher accuracy,
the number of iterations are often inhibitorily large, especially for
ill-conditioned problems.

Active set methods are fundamentally different from these approaches \cite{Bartlett2006, Goldfarb1983}.
They are designed to converge to the optimal point in a finite number
of iterations, up to the accuracy of round-off errors. They do this
by iteratively improving a guess, the \textit{working set}, of the
set of active constraints at the optimum, until the correct solution
is found. The set of active constraints at each iteration is referred
to as the \textit{active set}. The number of working
sets that needs to be tested therefore usually grows quickly with
the number of inequalities in the problem. In this paper, we focus
on an active set method, where the working set is updated by either
adding or removing one constraint at each iteration. Other approaches where multiple constraints are modified exist, and we believe our method can be used in such schemes, but that lies beyond the scope of this work.

The method we present is applying the active set method to a form arising when formulating the dual of a standard quadratic program.
By using a dual active set method, the main iterations of the algorithm, and the factorization
that needs to be updated, will scale with the set of constraints instead of the number of primal variables.
However, when there are linearly dependent constraints,
the dual will not have a positive definite quadratic cost,
which requires extra care in the solver.

\mnote{Much of the remaining par can probably be clarified.}
At each iteration, active set methods seeks to decrease the cost function
given the constraints of the current working set.
This sub-problem is posed as minimizing the quadratic function, subject to equality constraints.
When the problem is positive definite, then so is this sub-problem,
and a unique minimizer exists.
However, in the semi-definite case, these sub-problems can be semi-definite and even unbounded.
One approach to handle this problem is to ensure that the active set is always modified in a way that
keeps the sub-problems bounded.
Another approach is to allow the sub-problems to be unbounded, and in this case,
find a descent-direction of zero curvature~\cite{Wong2011}.
Thus at each iteration, it has to be detected if the problem is unbounded,
and then a corresponding method, to either solve for a minimizer or a descent direction,
can be applied.

In our approach however, by using iterative refinement to solve the sub-problems,
we are able to use the same algorithm to solve the bounded and unbounded case without first
determining if the sub-problem is unbounded.

Although regularization and iterative refinement are used in other algorithms to overcome problems with
semi-definite and ill-conditioned hessians~\cite{Potschka2010}\cite{Ferreau2014},
they still rely on methods to ensure that the sub-problems are bounded,
such as linear dependence tests when updating the working set.
As far as the authors know, this is the first time the same algorithm is used to solve both the consistent case,
and the case where the sub-problems are unbounded.

The second contribution is a different approach to factorizing the matrix needed
to solve the sub-problems,
which is independent in size of the number of constraints in the working set. These two contributions work well together, but can be used independently of each other.

Although a big motivation for active-set methods is their ability of \emph{warm start},
i.e. reuse the factorization and solution from a previous similar problem,
we will not focus on that property in this article.
Since the main outline of our algorithm is the same as previous approaches,
existing techniques for warm starting also applies to our method.

We do however present a very simple approach to selecting an initial guess of the active set in Section \ref{sec:initialpoint},
based on the simple form of the dual problem.
This way of selecting an initial guess proves very powerful in the numerical examples
in Section \ref{sec:numericalexamples}.
\subsection{Notation}

For a vector $v$ we denote the $i$:th element by $\left(v\right)_{i}$,
and for a matrix $A$, $\left(A\right)_{i,j}$ denotes the element
at row $i$, column $j$. Inequalities between vectors should be interpreted
as element wise inequality. For a finite set $\mathcal{W}$, $|\mathcal{W}|$ is the number of elements in the set, $(\mathcal{W})_{i}$
denotes the $i$:th element, given some arbitrary but consistent throughout
the article, ordering.

\section{Problem}

Consider the general mixed constraint quadratic program
\begin{eqnarray}
\min &  & \frac{1}{2}x^{T}Px+q^{T}x\label{eq:primal}\\
\text{s.t.} &  & Ax=b\nonumber\\
 &  & Cx\leq d\nonumber
\end{eqnarray}
with $x\in\mathbb{R}^{n}$,
$A\in\mathbb{R}^{m_\textrm{eq}\times n}$,
$C\in\mathbb{R}^{m_\textrm{in}\times n}$, where we assume that the matrix $P\in\mathbb{R}^{n\times n}$ is symmetric and positive definite, and that there exists at least one feasible point.
The resulting dual problem is
\begin{eqnarray}
\min_{\mu_{\textrm{eq}},\mu_{\textrm{in}}} &  & \frac{1}{2}\begin{bmatrix}\mu_{\textrm{eq}}\\
\mu_{\textrm{in}}
\end{bmatrix}^{T}\begin{bmatrix}AP^{-1}A^{T} & AP^{-1}C^{T}\\
CP^{-1}A^{T} & CP^{-1}C^{T}
\end{bmatrix}\begin{bmatrix}\mu_{\textrm{eq}}\\
\mu_{\textrm{in}}
\end{bmatrix}\label{eq:dualAC}\\
 &  & +\begin{bmatrix}AP^{-1}q+b\\
CP^{-1}q+d
\end{bmatrix}^{T}\begin{bmatrix}\mu_{\textrm{eq}}\\
\mu_{\textrm{in}}
\end{bmatrix}\nonumber \\
\text{s.t.} &  & \mu_{\textrm{in}}\geq0\nonumber,
\end{eqnarray}
where $\mu_\textrm{eq}\in\mathbb{R}^{m_\textrm{eq}}$ and $\mu_\textrm{in}\in\mathbb{R}^{m_\textrm{in}}$ are the dual variables for the equality and inequality constraints, respectively.
The minimum of the dual is attained by strong duality \cite[p.~226]{Boyd2010} and the
primal solution $x^*$ is given by the KKT conditions as
\begin{equation}
x^{*}=-P^{-1}(q+A^{T}\mu_{\textrm{eq}}^{*}+C^{T}\mu_{\textrm{in}}^{*\mathcal{}}).\label{eq:dualtoprimal}
\end{equation}

\section{Active set method}

\mnote{In this section we and outline the general approach to solve the dual problem using an active-set method.}
We now focus on solving the dual problem (\ref{eq:dualAC}), since
a solution $\mu^{*}$ to this problem can be used to simply find a
solution $x^{*}$ to the primal problem (\ref{eq:primal}) by solving (\ref{eq:dualtoprimal}). We implement
the standard active-set method as described in \cite{Nocedal2006}.
Since the dual problem \eqref{eq:dualAC} might not be positive definite,
we modify the algorithm to handle semi-definite problems. To simplify
the notation, let the dual problem be
\begin{align}
\min_{\mu}\quad & \frac{1}{2}\mu^{T}G\mu+\mu^{T}h\label{eq:dualproblem}\\
\text{s.t.}\quad & \mu_{\textrm{in}}\ge0,\nonumber
\end{align}
where
\begin{align*}
    G=\begin{bmatrix}AP^{-1}A^{T} & AP^{-1}C^{T}\\
    CP^{-1}A^{T} & CP^{-1}C^{T}
    \end{bmatrix}, \quad
    h = \begin{bmatrix}AP^{-1}q+b\\
   CP^{-1}q+d
   \end{bmatrix},
\end{align*}
and $\mu^{T}:=\begin{bmatrix}\mu_{\textrm{eq}}^{T}& \mu_{\textrm{in}}^{T}\end{bmatrix}$.
We define the set of indices corresponding to $\mu_{\textrm{in}}$ in $\mu$ as
$\mathcal{I}=\left\{m_\textrm{eq}+1,\dots,m_\textrm{eq}+m_\textrm{in}\right\}$.
Let $\mathcal{W}_{k}\subseteq\mathcal{I}$, the \textit{working set}
at iteration $k$, be the current guess of the active set at the solution
$\mu^{*}$, i.e the set so that $\mu$ at iteration $k$ satisfies $(\mu)_i=0,\forall i\in \mathcal{W}_k$.
At each iteration of the algorithm, a new point $\mu_{k+1}$ is generated by decreasing
the cost function, given the constraints defined by the working set.
If we let $\mu_{k+1}:=\mu_{k}+p_{k}$, this corresponds to finding a descent direction $p_{k}$, such that $(p_{k})_{i\in\mathcal{W}_{k}}=0$.
\mnote{The sub-problem in the algorithm consists of finding a descent direction $p_{k}$ from $\mu_{k}$ that remains in
the working set, i.e. $\mu_{k+1}=\mu_{k}+p_{k}$ being the minimizer
given $(\mu_{k+1})_{i\in\mathcal{W}_{k}}=0$.}
Substituting $\mu_{k+1}$ for $\mu$ in equation \eqref{eq:dualproblem} leads to the problem
\begin{align}
\min_{p_{k}}\quad & \frac{1}{2}p_{k}^{T}Gp_{k}+p_{k}^{T}\left(h+G\mu_{k}\right)\label{eq:subproblem1}\\
\text{s.t.}\quad & (p_k)_{i}=0\quad\forall i\in\mathcal{W}_{k},\nonumber
\end{align}
with KKT conditions
\begin{equation}
\begin{bmatrix}G & \mathcal{A}_{k}^{T}\\
\mathcal{A}_{k} & 0
\end{bmatrix}\begin{bmatrix}p_{k}\\
\lambda
\end{bmatrix}=\begin{bmatrix}-h-G\mu_{k}\\
0
\end{bmatrix},\label{eq:subproblem1kkt}
\end{equation}
where $\lambda\in\mathbb{R}^{|\mathcal{W}_k|}$ and $\mathcal{A}_{k}\in\mathbb{R}^{|\mathcal{W}_k|\times(m_\textrm{eq}+m_\textrm{in})}$ is the indicator matrix for the indices in
$\mathcal{W}_{k}$, i.e., with some abuse of notation, $(\mathcal{A}_{k})_{i,j}=1$
if $\left(\mathcal{W}_{k}\right)_{i}=j$, and $0$ otherwise. An overview
of the algorithm is presented in Algorithm \ref{alg:activeset}.

 \begin{algorithm}
\caption{Active-set method for solving problem \eqref{eq:dualproblem}}
\begin{algorithmic}[1]\label{alg:activeset}
\renewcommand{\algorithmicrequire}{\textbf{Input:}}
\renewcommand{\algorithmicensure}{\textbf{Output:}}
\ENSURE Solution $\mu^*$ to problem \eqref{eq:dualproblem}
\STATE Compute a feasible starting point $\mu_0$ (e.g $\mu_0=0$)
\STATE Let $\mathcal{W}_{0}$ be a set of active constraints at $\mu_{0}$
\FOR {$k = 0,1,2,...$}
\STATE \textit{find $p_{k}$ according to :}
\IF {(\ref{eq:subproblem1}) is bounded}
\STATE find a minimizing $p_{k}$ to (\ref{eq:subproblem1}) \label{lst:line:minimizer}
\ELSE
\STATE find a $p_{k}$ with negative cost such that $p_k^TGp_k=0$ \label{lst:line:nominimizer}
\ENDIF
\IF {$p_{k}=0$ \COMMENT{From line \ref{lst:line:minimizer}}}
\STATE Find Lagrange multipliers $\lambda^{*}$ from (\ref{eq:subproblem1kkt})
\IF {$(\lambda^{*})_{i}\geq0$ for all $i\in\mathcal{W}_{k}$}
\RETURN $\mu^{*} \gets \mu_{k}$
\ELSE
\STATE $j\gets\arg\min_{j\in\mathcal{W}_{k}}(\lambda^{*})_{j}$
\STATE $\mu_{k+1}\gets \mu_{k}$
\STATE $\mathcal{W}_{k+1}\gets\mathcal{W}_{k}\setminus\text{\ensuremath{\left\{  j\right\} } }$
\ENDIF
\ELSE
\IF {$p_{k}$ was minimizing (bounded)}
\STATE find the largest $\alpha_{k}\leq1$ so that $\mu_{k}+\alpha_{k}p_{k}$ is feasible\label{lst:line:akbounded}
\ELSE
\STATE find the largest $\alpha_{k}$ so that $\mu_{k}+\alpha_{k}p_{k}$ is feasible\label{lst:line:akunbounded}
\ENDIF
\IF {constraints were blocking $\alpha_{k}$ \COMMENT{line \ref{lst:line:akbounded},\ref{lst:line:akunbounded}}}
\STATE $\mathcal{W}_{k+1}\gets\mathcal{W}_{k}\cup\text{\ensuremath{\left\{  j\right\} } }$  \COMMENT{$j$ is a blocking constraint}
\ELSE
\STATE $\mathcal{W}_{k}\gets\mathcal{W}_{k}$
\ENDIF
\ENDIF
\ENDFOR
\end{algorithmic}
\end{algorithm}

The difference from \cite{Nocedal2006} in Algorithm \ref{alg:activeset}
are lines \ref{lst:line:nominimizer} and \ref{lst:line:akunbounded}
which handle the case where the dual sub-problem (\ref{eq:subproblem1})
is unbounded. Line \ref{lst:line:akunbounded} assumes that there
exists a largest $\alpha_{k}$. Because $p_{k}$ is a descent direction
of zero curvature, an unbounded $\alpha_{k}$ would mean that the
dual is unbounded which contradicts the strong duality and that a primal feasible point exists.

Note that the first step of finding a feasible
point $\mu_0$ is trivial, e.g $\mu_0=0$ is feasible,
since we only have non-negative inequality
constraints and no equality constraints.
A discussion of finding a better initial guess is discussed in Section \ref{sec:initialpoint}.

In the following sections we will present how to solve the sequence
of sub-problems (\ref{eq:subproblem1}) in an efficient way. The main
contribution is that we can use a single factorization and algorithm
to solve both the case of finding a minimizer as well as finding a
descent direction of zero curvature, using iterative refinement with cheap
updates to the factorization at each step.

\subsection{Factorization}

To solve the sub-problems in Algorithm \ref{alg:activeset} at lines \ref{lst:line:minimizer} and \ref{lst:line:nominimizer} a factorization of the quadratic term is needed.
We begin by assuming that the columns in $\left[A\,\, C\right]$ are linearly independent.
This will be relaxed in Section \ref{sec:linearlydependent}. The matrix $G$
is therefore positive definite and allows for a Cholesky factorization
$G=LL^{T}.$

The crucial step in the active set method is solving the sub-problem \eqref{eq:subproblem1}
of the form
\begin{eqnarray}
\min_{\mu} &  & \frac{1}{2}p^{T}Gp+c^{T}p\label{eq:dualactiveproblem}\\
\text{s.t} &  & (p)_{i}=0,\quad i\in\mathcal{W}_{k},\nonumber
\end{eqnarray}
where the indices in the working set $\mathcal{W}_{k}$ will be indices
corresponding to the constraints on $\mu_{2}$.
When $G$ is positive definite, there is a unique solution to this
problem and finding it is equivalent to solving the KKT system
\begin{equation}\label{eq:kktdual}
\begin{bmatrix} G & \mathcal{A}_{k}^{T}\\
\mathcal{\mathcal{A}}_{k} & 0
\end{bmatrix}\begin{bmatrix}p\\
\lambda
\end{bmatrix}=\begin{bmatrix}-c\\
0
\end{bmatrix},
\end{equation}
where $\mathcal{A}_{k}$ is the indicator matrix, i.e. $(\mathcal{A}_{k})_{i,j}=1$
if $(\mathcal{W}_{k})_{i}=j$ and $0$ otherwise.
Now, note that finding $p$
in (\ref{eq:dualactiveproblem}) is equivalent to solving the problem
\begin{eqnarray*}
\min &  & \frac{1}{2}p^{T}\bar{G}p+\bar{c}^{T}p\\
\text{s.t} &  & (p)_{i}=0,\quad i\in\mathcal{W}_{k}
\end{eqnarray*}
where $\bar{G}$ is a modified version of $G$ with identity mapping
for indices $i\in\mathcal{W}_{k}$, i.e. the new matrix $\bar{G}$
in terms of $G$ is
\begin{equation}\label{eq:gbar}
(\bar{G})_{i,j}=\begin{cases}
1 & \text{if }i=j\in\mathcal{W}_{k}\\
0 & \text{if i\ensuremath{\neq j} and}\,i\in\mathcal{W}_{k}\,\text{or }j\in\mathcal{W}_{k}\\
(G)_{i,j} & \text{otherwise}
\end{cases}
\end{equation}
and
\[
(\bar{c})_i=\begin{cases}
0 & \text{if } i\in\mathcal{W}_k\\
(c)_i & \text{else.}
\end{cases}
\]
Problem (\ref{eq:dualactiveproblem}) can therefore be solved from
\begin{align*}
\bar{G}p & =-\bar{c}
\end{align*}
instead, and the dual variable $\lambda$ in \eqref{eq:kktdual} can be calculated as
\[
(\lambda)_{j}=(-Gp-c)_{i}\text{ \text{for }}i=(\mathcal{W}_{k})_{j}.
\]
If $G$ is positive definite, then so is $\bar{G}$ by Lemma~\ref{lem:posdef} in the Appendix.
The linear solution can therefore be computed efficiently if a Cholesky factorization of $\bar{G}$
is available. As we show in the appendix, updating the Cholesky
factorization $\bar{L}\bar{L}^{T}=\bar{G}$ to $\tilde{L}\tilde{L}^{T}=\tilde{G}$,
where a single element is added (or removed) to the working set $\mathcal{W}_{k}$
can be reduced to a rank-$1$ update (down-date) of the Cholesky factorization.
Since the matrix $\bar{G}$ is positive definite, regardless of the
active set, these updates are well behaved operations\mnote{SOME CITATION
OR RESULT}, and the rank-$1$ update can be done in $O(n^{2})$ operations.
\mnote{(ACTUALLY $O(m^{2}+n)$ where $m$ is the
set of inequalities with indices after $k$ for update, and $O\left((n-\left|\mathcal{W}_{k}\right|)^{2}+n\right)$
for down-date, where $\left|\mathcal{W}_{k}\right|$ is the number
of indices in the working set.)} This allows us to work with a matrix
that is considerably smaller than the full KKT system while keeping the size of the factorized matrix $\bar{G}$ independent of the working set $\mathcal{W}_{k}$,
allowing efficient memory usage. This approach is possible in the
dual because of the simple form of the constraints; $\mu_{2}\ge0$,
but could easily be adapted for the slightly more general form $\mu_{2}\ge v$.

A common alternative to this factorization is to instead work with a \emph{reduced hessian}.
This corresponds to working with a hessian that is defined on the null-space of $\mathcal{A}_k$.
This means that a factorization that reveals the null-space of $\mathcal{A}_k$ is needed.
When working with the special form of the dual, this approach would be similar to ours,
but the size of the factorized matrix would vary with the size of the active set.

\subsection{Iterative refinement for solving a linear system or its null-space
projection}

\mnote{Rethink where to place this}
To solve the general problem where the dual is positive semi-definite,
we start by analyzing the method known as iterative refinement using
some tools from monotone operator theory. The linear system
\begin{align*}
\text{find }x:\,\, & Ax=b
\end{align*}
 is equivalent to finding a point $0 = F(x)$ where $F(x)=Ax-b$. The
resolvent $J_{\gamma F}=\left(\gamma F+I\right)^{-1}$ is known to
be firmly non-expansive if and only if $F$ is monotone \cite[Prop 23.7]{Bauschke2011}. Moreover
$F$ is monotone if and only if $A$ is positive semi-definite \cite[Ex 20.15]{Bauschke2011}.
The proximal point algorithm
\[
x_{k+1}=J_{\gamma F}x_{k}
\]
or equivalently,
\[
x_{k+1}=\arg\min_{x}\left(\frac{1}{2}x^TAx-b^Tx+\frac{1}{2\gamma}\left\Vert x-x_{k}\right\Vert ^{2}\right),
\]
is known to converge to a point $x^{*}$ satisfying $0\in F(x^{*})$ when
$F$ is monotone and such a point $x^{*}$ exists \cite[Thm 23.41]{Bauschke2011}. This method
can be used to get high accuracy solutions to linear systems, especially
when the problem is ill-conditioned or singular.

We now show what happens when there is no solution to $Ax=b$.\mnote{, which relates closely to to the infimal-displacement vector for inconsistent problem when using the Douglas-Rachford algorithm \cite{Bauschke2016a}. }\mnote{WHAT MORE?}
This result proves very useful when the dual problem is semi-definite as
seen in the next section.
\begin{thm}
\label{thm:descent}Let $F(x)=Ax-b$ with $A$ symmetric positive
semi-definite. Assume that there is no $x^{*}$ such that $0=F(x^{*}).$
The iterative refinement
\[
x_{k+1}=J_{\gamma F}x_{k}
\]
will result in a sequence where
\[
(x_{k+1}-x_{k}) \rightarrow-\gamma b_{\textrm{N}},
\]

where $b_{\textrm{N}}$ is the projection of $b$ onto the null-space of $A$.
\end{thm}
\begin{proof}
The resolvent $J_{\gamma F}$ is firmly non-expansive for positive
definite $A$ as explained above. For firmly non-expansive
$T$, the algorithm $x_{k+1}=Tx_{k}$ has the property that
\[
(x_{k+1}-x_{k}) \rightarrow \delta x
\]
where $\delta x$ is the unique minimum norm element in $\overline{\text{range}}\left(I-T\right)$
\cite[Cor 2.3]{Bailion}\cite[Fact 3.2]{Bauschke2004}. Letting $T=J_{\gamma F}$, we first calculate an
expression for $I-T$. Rewriting the proximal point algorithm and
substituting $\epsilon=1/\gamma$ gives
\begin{align*}
x_{k+1} & = \left(\gamma F+I\right)^{-1}x_{k} & & \Leftrightarrow\\
\left(\gamma F+I\right)x_{k+1} & =x_{k} & & \Leftrightarrow\\
(\gamma A+I)x_{k+1} & =x_{k}+\gamma b & & \Leftrightarrow\\
x_{k+1} & =\left(A+\frac{1}{\gamma}I\right)^{-1}\left(\frac{1}{\gamma}x_{k}+b\right) & & \Leftrightarrow\\
x_{k+1} & =\left(A+\epsilon I\right)^{-1}\left(\epsilon x_{k}+Ax_{k}-Ax_{k}+b\right) & & \Leftrightarrow\\
x_{k+1} & =\left(A+\epsilon I\right)^{-1}\left(b-Ax_{k}\right)+x_{k}
\end{align*}
i.e. $(I-T)(x)=\left(A+\epsilon I\right)^{-1}(Ax-b)$. Since $I-T$
is an affine function in $\mathbb{R}^{n}$, its range is closed and
we set out to find the minimum norm element. Let $y\in\text{range}\left(I-T\right)$,
then for some $x$ we have
\begin{equation}
y=(I-T)(x)=\left(A+\epsilon I\right)^{-1}(Ax-b).\label{eq:deltax}
\end{equation}
Let $\textrm{N}(A)$ and $\textrm{R}(A)$ denote the null and range-space of $A$.
From symmetry of $A$ we have $\textrm{N}(A)^{\perp}=\textrm{R}(A^{T})=\textrm{R}(A)$
so $A$ is bijective on $\textrm{R}(A)\rightarrow \textrm{R}(A)$,
and thus so is $(A+\epsilon I)$
and its inverse. Let $x=x_{\textrm{N}}+x_{\textrm{R}}$, where $x_{\textrm{N}}\in \textrm{N}(A),\,x_{\textrm{R}}\in \textrm{R}(A)$
and similarly for $b$ and $y$. Equation (\ref{eq:deltax}) can then be split
to the parts in the range and null-space, i.e.
\begin{align*}
y_{\textrm{N}} & =\left(A+\epsilon I\right)^{-1}(Ax_{\textrm{N}}-b_{\textrm{N}})\\
y_{\textrm{R}} & =\left(A+\epsilon I\right)^{-1}(Ax_{\textrm{R}}-b_{R})
\end{align*}
where $y=y_{\textrm{N}}+y_{\textrm{R}}$. The first equation gives
\[
Ay_{\textrm{N}}+\epsilon y_{\textrm{N}}=-b_{\textrm{N}}\Longrightarrow y_{\textrm{N}}=-\frac{1}{\epsilon}b_{\textrm{N}}.
\]
Since $A$ and $(A+\epsilon I)^{-1}$ are bijective on $\textrm{R}(A)\rightarrow \textrm{R}(A)$,
i.e any $y_{\textrm{R}}$ can be reached from $x_{\textrm{R}}$, the second equation
gives that
\[
\text{range}(I-T)=\text{\ensuremath{\left\{  y=y_{\textrm{N}}+y_{\textrm{R}}\mid y_{\textrm{N}}=-\frac{1}{\epsilon}b_{\textrm{N}},\,y_{\textrm{R}}\in \textrm{R}(A)\right\} } }.
\]
The minimum norm element $\delta x$ is thus given by $y_{\textrm{R}}=0,\,y_{\textrm{N}}=-\frac{1}{\epsilon}b_{\textrm{N}}$,
i.e.
\[
\delta x=-\frac{1}{\epsilon}b_{\textrm{N}}=-\gamma b_{\textrm{N}}..
\]
\end{proof}

\subsection{Semi-definite case}\label{sec:linearlydependent}
In the case where $\begin{bmatrix}A & C\end{bmatrix}$
has linearly dependent columns, the matrix $G$ will be positive-semi
definite. This is a common case, if for example $C$ encodes both
upper and lower bounds. This means that the minimization problem (\ref{eq:dualactiveproblem})
\begin{eqnarray*}
\min_{\mu} &  & \frac{1}{2}p^{T}Gp+c^{T}p\\
\text{s.t} &  & (p)_{i}=0,\quad i\in\mathcal{W}_{k},
\end{eqnarray*}
in the active set method could be unbounded.

The goal is then to instead find a direction $p$ in which the cost
is decreasing towards infinity, i.e. finding $p$ such that
\begin{align*}
p^{T}Gp & =0\\
(p)_{i} & =0,\,i\in\mathcal{W}_{k}.\\
c^{T}p & <0
\end{align*}
Since $G$ is symmetric, $p^{T}Gp=0$ if and only if $Gp=0$, so the
two first conditions are equivalent to $\bar{G}p=0$ where $\bar{G}$
is as described in the previous section.

The obvious choice here is to find the direction $p^{*}$ of maximal
descent, i.e. the projection of the linear part $-c$ onto the null-space
\begin{align}
p^*=\arg\min_{p}\, & \left\Vert p+c\right\Vert \label{eq:subspaceprojection}\\
\text{s.t } & \bar{G}p=0.\nonumber
\end{align}

There are a few alternatives to solving this problem in existing solvers.
One way is to use the more expensive QR decomposition,
which can reveal the null-space of $G$.
However, if iterative refinement is to be used, an additional factorization
of  $G+\epsilon I$ would also be needed. We now show that this is not needed with our approach.

Consider the problem \eqref{eq:subspaceprojection} above. If $p^{*}=0$, then $-c\in\textrm{R}\left(\bar{G}\right)$ and no
such descent direction exist (i.e. (\ref{eq:dualactiveproblem}) attains
it minimum). If $p^{*}\neq0$, then since $p^{*}$ is the orthogonal
projection onto the subspace $\text{N}(\bar{G})$ from $-c$, we have
$p^{*T}c=-p^{*T}p^{*}<0$, i.e $p^{*}$ is a direction of (maximal)
descent. But finding the projection of $-c$ onto the null-space of
$\bar{G}$ is precisely what the iterative refinement will achieve
when there is no solution to the problem $\bar{G}x=c$ as shown in
Theorem \ref{thm:descent}\mnote{clarify}. We therefore see that if we apply the
iterative refinement to the problem $\bar{G}x=c$, it will either
converge to the solution of the problem (\ref{eq:dualactiveproblem}),
or if no solution exist, the iterates will reveal the direction
of maximal descent.

\subsubsection{Factorization for semi-definite case}

The method of iterative refinement relies on solving the linear system
\[
\left(\bar{G}+\epsilon I\right)x=c
\]
multiple times. Instead of storing the factorization $\bar{G}=\bar{L}\bar{L}^{T}$
which might not exist when $\bar{G}$ is semi-definite, we store
the Cholesky factorization $\left(\bar{G}+\epsilon I\right)=\tilde{L}\tilde{L}^{T}$
instead.
Just as before, this factorization is simple to update when
the working set is changed.

\subsubsection{Detecting Solution or Maximal Descent}

Although the behavior of the iterative refinement is different depending
on whether the linear system has a solution or not, we need a way of detecting
it. Other approaches to solving the problem often struggle with differentiating
between if there is no solution or if the curvature is very low. Since
our approach factorizes the matrix $\left(\bar{G}+\epsilon I\right)$
we have a lower bound on the smallest eigenvalue, and the factorization
can be ensured to be robust. In our testing, $\epsilon\in(10^{-6},10^{-8})$
seems to be a good trade-off between robustness of the factorization
and convergence rate.\mnote{\textcolor{red}{DISCUSSION ON CONV RATE?}}\mnote{\textcolor{red}{LINEAR CONVERGENCE GIVES ddx RESULT}.}
Moreover,
the iterates will behave fundamentally different when there is a solution,
and when there is not. In the first case, the iterates will converge,
and in particular $\left\Vert x_{k+1}-x_{k}\right\Vert \rightarrow0$.
In the case where there is no solution, the difference $x_{k+1}-x_{k}\rightarrow-\frac{1}{\epsilon}b_{\textrm{N}}$
so both $\left\Vert x_{k+1}-x_{k}\right\Vert $ and $\left\Vert x_{k}\right\Vert $
will be very large.

\subsubsection{Convergence rate}
In both cases, the convergence rate of the iterative refinement, either of $x^k$ to a point $x^*$ or of the sequence $x^{k+1}-x^k$ to $-\gamma b_n$, is determined by the eigenvalues of the matrix $(\gamma\bar{G}+I)^{-1}$. The rate is given by the largest eigenvalue that is not $1$, i.e. $\frac{1}{\gamma\lambda_\textrm{min}+1}=\frac{\epsilon}{\lambda_\textrm{min}+\epsilon}$, where $\lambda_\textrm{min}$ is the smallest non-zero eigenvalue $\bar{G}$.
It is therefore important to select $\epsilon$ to be small enough in relation to the eigenvalues, without compromising the numerical accuracy of the factorization.

\subsubsection{Alternative approach using iterative refinement}

For ill-conditioned problems, in the case where $\bar{G}x=c$ lack
a solution, the convergence of $x_{k+1}-x_{k}$ to the projection
of $-c$ onto the null-space of $G$ might be relatively slow. Iterative refinement
can then be used to solve the projection problem directly by applying it to the equation $\bar{G}x=0$. The problem obviously
has a solution, moreover non-expansiveness of $(I+\gamma \bar{G})^{-1}$
implies that each step of the algorithm gets closer to all the points
in the set $\{x\mid\bar{G}x=0\}$. But this set is a subspace, so
$x_k\rightarrow x^{*}$ must then be the orthogonal projection onto $\bar{G}x=0$
from the initial point. Letting $x_{0}=-c$, thus solves
Problem
\ref{eq:subspaceprojection}.

An alternative initial point would be $x_{k+1}-x_{k}$, as obtained
after a couple of iterations of trying to solve $\bar{G}x=c$ using
iterative refinement. Although this method would not exactly give
the projection from $-c,$ but instead from $x_{k+1}-x_{k},$ it should
be a good approximation of the maximal descent direction, and it will
satisfy the zero-curvature condition $\bar{G}x=0$.

\section{Initial active set}\label{sec:initialpoint}
From the simple form of the dual problem \eqref{eq:dualproblem}
it is trivial to find a (dual) feasible point, e.g $\mu_0=0$.
However, a good initial guess of the active set at the solution can significantly reduce the
number of iterations, i.e changes to the working set,
needed to find the optimal point.

One approach would be to find a minimizer (if existent) to the unconstrained problem.
This would require an additional factorization of the quadratic term.
Instead, we look at the gradient of the cost function \eqref{eq:dualproblem} at the origin, i.e. $h$. For each coefficient pointing out from the feasible area, we set that constraint to being active. This gives us an initial guess of the active set at the solution as
\[
    \mathcal{W}_0 = \left\{ i \mid  (h)_i < 0, i\in \mathcal{I} \right\},
\]
which we refer to as ``smartstart'' in the numerical examples below.

\mnote{
\section{OTHER THOUGHTS}

\textit{\textcolor{red}{DISCUSS STABILITY OF DOING RANK UP/DOWN-DATES
SOMEWHERE? SEEMS TO BE WELL BEHAVED}}
}
\section{Numerical Examples}\label{sec:numericalexamples}

We apply the proposed algorithm to two different problems in this section,
and compare it to the active set solver qpOASES \cite{Ferreau2014}.
Our algorithm is implemented in the programming language Julia \cite{Bezanson2014},
and is open source and available on \texttt{github}~\cite{QPDAS}. As
a result of being written in Julia, the implementation is not only
fast, but allows for a wide range of different numerical types. The
main numerical results are run using Float64 (IEEE 754) for which
efficient BLAS implementations are used for the matrix factorizations
and operations, but the code supports types of arbitrary precision. \mnote{To illustrate the power and speed when using generic
types, we also present some results with 128bit arithmetic using the
efficient implementation in \texttt{DoubleFloats.jl }CITE, as well
as using the built in BigFloat in Julia which can be set to arbitrary
precision. In both these cases, Julia's built in generic algorithm
for Cholesky factorization and rank updates is used.}
The MPC example is chosen, not primarily to illustrate a case where we expect an active set
method to excel, but to illustrate that the algorithm is able to handle
even very ill-conditioned problems. The polytope projection algorithm
on the other hand is exactly the kind of problem where a dual active
set method is very efficient. The number of primal variables is large,
but the resulting dual problem is small.
Moreover, since $P=I$, recovering the primal solution
from the dual using Equation (\ref{eq:dualtoprimal}) is cheap.

\subsection{MPC Example}

\mnote{
SOME MPC BACKGROUND, ACTIVE SET METHODS (\textit{\textcolor{red}{qpOASES}}
\cite{Ferreau2014})}

To benchmark the algorithm, we consider the problem of controlling an AFTI-16 aircraft in the Model Predictive Control (MPC) setting, as in \cite{Bemporad1997,Giselsson2014}. The linear and discretized model of the system is given by
\begin{equation}\label{eq:dynamicsystem}
x[k+1] = Ax[k] + Bu[k],
\end{equation}
where
\[
A=\begin{bmatrix}0.999 & -3.008 & -0.113 & -1.608\\
0 & 0.986 & 0.048 & 0\\
0 & 2.083 & 1.009 & 0\\
0 & 0.053 & 0.050 & 1
\end{bmatrix},
B=\begin{bmatrix}-0.080 & -0.635\\
-0.029 & -0.014\\
-0.868 & -0.092\\
-0.022 & -0.002
\end{bmatrix}
\]
We formulate the MPC problem as minimizing
\[
J(x,u)=\Sigma_{k=1}^{N}x[k]^TQx[k]+u[k]^TRu[k].
\]
subject to $l_x\leq x[k] \leq u_x, x[0]=x_0$ and the dynamics \eqref{eq:dynamicsystem}.
Using the equation for the dynamics, the variables $x[k]$ can be eliminated, and the optimization problem can be written on \emph{reduced form}:
\begin{align}
    \min_{\bar{\mu}}\quad & \bar{u}^TF\bar{u} + 2\bar{u}^TGx_0 + x_0^THx_0\\
    \text{s.t.\quad} & l_{\bar{u}} \leq C\bar{u} \leq u_{\bar{u}},\nonumber
\end{align}
where $\bar{u}^T=\begin{bmatrix}u[1]^T & u[2]^T & \dots &u[N]^T\end{bmatrix}$.
For our tests, we let $N=30$, $Q=R=I$, $l_x=-u_x$, with $u_x^T=\begin{bmatrix}0.2 & 0.2 &0.2 &0.2\end{bmatrix}$. This gives a primal problem with $F\in\mathbb{R}^{60\times60}$, $C\in\mathbb{R}^{120\times60}$, where $F$ is positive definite with a condition number $\kappa(F)\approx 10^8$, which illustrates that the problem is very ill-conditioned.
Rewriting it again to form the dual \eqref{eq:dualAC} results in a problem with 240 variables, and a quadratic term with rank 60.

Whereas active set methods are well suited for MPC problems, where multiple similar problems are solved in sequence, we focus on the performance of solving a single problem.
The results of solving the problem with our method QPDAS, compared to qpOASES~\cite{Ferreau2014}, are presented in Table \ref{tab:mpc}, and were run on a standard desktop PC.
The results for qpOASES were obtained using its MATLAB interface.
The two results for our algorithm are presented both with and without the ``smartstart''
from Section \ref{sec:initialpoint}, and includes the time to recover the primal solution.
We also present the number of iterations of iterative refinement that was used at each iteration.
For qpOASES, the three cases (primal 1), (primal 2) and (dual), correspond to the cases where qpOASES was given either
(i) the primal problem with inequalities encoded as upper bounds,
(ii) encoded as upper and lower bounds, or
(iii) the dual problem.

\begin{table}
\vspace{0.15cm}
\caption{\label{tab:mpc}MPC example}
\centering
\begin{tabular}{|c|c|c|c|}
\hline
\textbf{Method} & \textbf{Time} & \textbf{Iterations} & \textbf{Refinement} \tabularnewline
& & & \textbf{Iterations}  \tabularnewline
\hline
\hline
QPDAS (primal) & 167ms & 258 & $3-6$\tabularnewline
\hline
QPDAS (primal smartstart) & 24ms & 35 & $3-6$\tabularnewline
\hline
qpOASES (primal 1) & 12ms & 90 & -\tabularnewline
\hline
qpOASES (primal 2) & 9.8ms & 90 & -\tabularnewline
\hline
qpOASES (dual) & 5.3ms & 24 & -\tabularnewline
\hline

\end{tabular}
\end{table}

\subsection{Polytope Projection}

We consider the problem of projecting a point $c\in\mathbb{R}^{n}$
onto a polytope described by a set of equalities $Cx\leq d$, where $C\in \mathbb{R}^{m\times n}$ and $m$ is much smaller than $n$. This is a case where the dual problem will be much smaller than the primal and thus very well suited for a dual-active set method.  Moreover, recovering the primal solution using equation \eqref{eq:dualtoprimal} will be very cheap, since the quadratic term $P$ is identity. The total cost of recovering the primal solution from the dual therefore consists of a matrix multiplication and a vector addition. The results are presented for two cases, $n=1000, m=50$ and $n=10000, m=500$. The inequality constraints were generated randomly so that approximately half of the constraints were active at the optimal point. The tests were run in the same way as for the MPC example and are presented in Table \ref{tab:polytope1} and \ref{tab:polytope2}.

Since the dual problem is considerably smaller than the primal, the cost of recovering the primal solution is still noticeable.
The cost for solving the dual, excluding the cost of recovering the primal is therefore presented as (dual). This enables a fair comparison between our method and qpOASES. The results for the primal problem with qpOASES were run with the auxiliary input \texttt{hessianType=1} to indicate that the quadratic matrix is identity, to avoid supplying a full matrix.

\begin{table}
\vspace{0.15cm}
\caption{\label{tab:polytope1}Polytope Projection, $n=1000, m=50$}
\centering
\begin{tabular}{|c|c|c|c|}
\hline
\textbf{Method} & \textbf{Time} & \textbf{Iterations} & \textbf{Refinement}\tabularnewline
 &  &  & \textbf{Iterations}\tabularnewline
\hline
\hline
QPDAS (primal) & 1.7ms & 25 & $3-6$\tabularnewline
\hline
QPDAS (primal smartstart) & 0.86ms & 2 & $3-6$\tabularnewline
\hline
QPDAS (dual) & 0.79ms & 25 & $3-6$\tabularnewline
\hline
QPDAS (dual smartstart) & 0.12ms & 2 & $3-6$\tabularnewline
\hline
qpOASES (primal) & 12s & 1071 & -\tabularnewline
\hline
qpOASES (dual) & 0.48ms & 31 & -\tabularnewline
\hline
\end{tabular}\\
\vspace{0.7cm}
\caption{\label{tab:polytope2}Polytope Projection, $n=10000, m=500$}
\centering
\begin{tabular}{|c|c|c|c|}
\hline
\textbf{Method} & \textbf{Time} & \textbf{Iterations} & \textbf{Refinement}\tabularnewline
 &  &  & \textbf{Iterations}\tabularnewline
\hline
\hline
QPDAS (primal) & 750ms & 245 & $3-7$\tabularnewline
\hline
QPDAS (smartstart) & 203ms & 39 & $3-7$\tabularnewline
\hline
QPDAS (dual) & 613ms & 245 & $3-7$\tabularnewline
\hline
QPDAS (dual smartstart) & 92ms & 39 & $3-7$\tabularnewline
\hline
qpOASES (primal) & 11hours & 11333  & -\tabularnewline
\hline
qpOASES (dual) & 270ms & 242 & -\tabularnewline
\hline
\end{tabular}
\end{table}

\section{Conclusions}
We have presented an active set algorithm for solving quadratic programming
problems of the form \eqref{eq:dualproblem}.
The method requires a single factorization to solve both sub-problems that
arise in a standard active-set approach,
and is designed to be numerically robust.
Together with a simple rule for selecting the initial working set,
the algorithm is able to solve problems with a few number of inequalities
extremely efficiently.

\mnote{COMMENT ON USING APPROPRIATE FACTORIZATION OF $P$ ?}
\appendix{}

\subsubsection*{Adding a constraint}

According to the discussion above, adding a constraint for index $i$
corresponds to creating an identity mapping in $\bar{G}$ for the
corresponding index. Let $\bar{G}_{k}=\bar{L}\bar{L}^{T}$ be the
matrix before the update, and $G_{k+1}=LL^{T}$ after, we get the
following relations
\[
\bar{G}_{k}=\begin{bmatrix}\bar{G}_{11} & \bar{G}_{12} & \bar{G}_{13}\\
\bar{G}_{12}^{T} & \bar{G}_{22} & \bar{G}_{23}\\
\bar{G}_{13}^{T} & \bar{G}_{23}^{T} & \bar{G}_{33}
\end{bmatrix}=\begin{bmatrix}\bar{L}_{11} & 0 & 0\\
\bar{L}_{21} & \bar{\ell}_{22} & 0\\
\bar{L}_{31} & \bar{L}_{32} & \bar{L}_{33}
\end{bmatrix}\begin{bmatrix}\bar{L}_{11} & 0 & 0\\
\bar{L}_{21} & \bar{\ell}_{22} & 0\\
\bar{L}_{31} & \bar{L}_{32} & \bar{L}_{33}
\end{bmatrix}^{T}
\]

\[
G_{k+1}=\begin{bmatrix}G_{11} & 0 & G_{13}\\
0 & 1 & 0\\
G_{13}^{T} & 0 & G_{33}
\end{bmatrix}=\begin{bmatrix}L_{11} & 0 & 0\\
L_{21} & \ell_{22} & 0\\
L_{31} & L_{32} & L_{33}
\end{bmatrix}\begin{bmatrix}L_{11}^{T} & L_{21}^{T} & L_{31}^{T}\\
 & \ell_{22} & L_{32}^{T}\\
 &  & L_{33}^{T}
\end{bmatrix}
\]
with $L_{11}=\bar{L}_{11}$, $L_{21}=0$, $L_{31}=\bar{L}_{31}$,
$L_{32}=0$, $\ell_{22}=1$ we see that
\begin{align*}
\bar{G}_{33} & =\bar{L}_{31}\bar{L}_{31}^{T}+\bar{L}_{32}\bar{L}_{32}^{T}+\bar{L}_{33}\bar{L}_{33}^{T}\\
G_{33} & =\bar{L}_{31}\bar{L}_{31}^{T}+L_{33}L_{33}^{T}
\end{align*}
and since $\bar{G}_{33}=G_{33}$ we get
\[
L_{33}L_{33}^{T}=\bar{L}_{33}\bar{L}_{33}^{T}+\bar{L}_{32}\bar{L}_{32}^{T}
\]
where $\bar{L}_{32}$ is a column vector.
\newpage
This corresponds to a rank-one
update of $\bar{G}_{k}$ with $\bar{L}_{32}\bar{L}_{32}^{T}$, either
directly of $L_{33}$ , or of
$\bar{L}$ with $\begin{bmatrix}\boldsymbol{0}\\
\bar{L}_{32}
\end{bmatrix}\begin{bmatrix}\boldsymbol{0}\\
\bar{L}_{32}
\end{bmatrix}^{T}$ where $\boldsymbol{0}$ is a column with $n-i+1$ zeros, where $i$
is the row and column being updated. This update requires $\mathcal{O}(n)^{2}$
operations. The corresponding update from $(\bar{G}_{k}+\epsilon I)=\bar{L}\bar{L}^{T}$
to $(G_{k+1}+\epsilon I)=LL^{T}$ follows correspondingly.

\subsubsection*{Removing a constraint}

Removing a constraint corresponds to reversing the process described
above. The equations are given by
\[
{\scriptscriptstyle \bar{G}_{k}=\begin{bmatrix}\bar{G}_{11} & 0 & \bar{G}_{13}\\
0 & 1 & 0\\
\bar{G}_{13}^{T} & 0 & \bar{G}_{33}
\end{bmatrix}=\begin{bmatrix}\bar{L}_{11} & 0 & 0\\
0 & 1 & 0\\
\bar{L}_{31} & 0 & \bar{L}_{33}
\end{bmatrix}\begin{bmatrix}\bar{L}_{11} & 0 & 0\\
0 & 1 & 0\\
\bar{L}_{31} & 0 & \bar{L}_{33}
\end{bmatrix}^{T}}
\]
\[
{\scriptscriptstyle G_{k+1}=\begin{bmatrix}G_{11} & G_{12} & G_{13}\\
G_{12}^{T} & G_{22} & G_{23}\\
G_{13}^{T} & G_{23}^{T} & G_{33}
\end{bmatrix}=\begin{bmatrix}L_{11} & 0 & 0\\
L_{21} & \ell_{22} & 0\\
L_{31} & L_{32} & L_{33}
\end{bmatrix}\begin{bmatrix}L_{11}^{T} & L_{21}^{T} & L_{31}^{T}\\
 & \ell_{22} & L_{32}^{T}\\
 &  & L_{33}^{T}
\end{bmatrix}^{T}}
\]
We get $L_{11}=\bar{L}_{11}$, $L_{31}=\bar{L}_{31}$

\begin{align*}
G_{12}=L_{11}L_{21}^{T} & \Rightarrow L_{21}^{T}=L_{11}\backslash H_{12}\\
G_{22}=L_{21}L_{21}^{T}+\ell_{22}\ell_{22} & \Rightarrow\,\,\ell_{22}=\sqrt{H_{22}-L_{21}L_{21}^{T}}\\
G_{23}=L_{21}L_{31}^{T}+\ell_{22}L_{32}^{T} & \Rightarrow\,L_{32}^{T}=\left(H_{23}-L_{21}L_{31}^{T}\right)/\ell_{22}\\
\bar{G}_{33}= & \bar{L}_{31}\bar{L}_{31}^{T}+\bar{L}_{33}\bar{L}_{33}^{T}\\
G_{33}= & \bar{L}_{31}\bar{L}_{31}^{T}+L_{32}L_{32}^{T}+L_{33}L_{33}^{T}
\end{align*}

but $\bar{G}_{33}=G_{33}$ so

\[
L_{33}L_{33}^{T}=\bar{L}_{33}\bar{L}_{33}^{T}-L_{32}L_{32}^{T},
\]
i.e.~a rank-one down-date of $\bar{L}_{33}$ with $L_{32}L_{32}^{T}$.
The down-date requires $\mathcal{O}(n)^{2}$ operations, the same
is true for the triangular back-solve and the rest are vector operations.
%
%
\begin{lem}\label{lem:posdef}
    If $G$ is positive definite, then so is $\bar{G}$, as defined in equation \eqref{eq:gbar}.
\end{lem}
\begin{proof}
    Let $S=\{x\mid\,(x)_i = 0\,\forall i\in\mathcal{W}_k\}$, and assume that $G$ is positive definite. We consider two cases: If $x\in S$ with $x\neq0$, then $0<x^TGx=x^T\bar{G}x$. If $x\not\in S$ with $x\neq0$, then $x^T\bar{G}x= x^T\tilde{G}x + \sum_{i\in\mathcal{W}_k}x_i^2$, where
    $\tilde{G}$ is the matrix $G$ with rows and columns $i\in\mathcal{W}_k$ set to zero. From $G$ being positive definite, $\tilde{G}$ must be positive semi-definite, i.e. $x^T\tilde{G}x\geq 0$. And since $x\not\in S$ we get $\sum_{i\in\mathcal{W}}x_i^2>0$.
    Thus for all $x\neq0$ we get $x^T\bar{G}x>0$.
\end{proof}
\bibliographystyle{plain}
\bibliography{biblo}

\end{document}